%% file: measures-and-slaloms-160411-pre.tex
\documentclass[12pt]{amsart}

\usepackage{amssymb,amsfonts}
\usepackage{amsmath,amsthm,amscd}
\usepackage{mathrsfs} 
\usepackage{enumerate}
\usepackage{verbatim}
\usepackage{color}
\usepackage[colorlinks]{hyperref}
\usepackage[displaymath, tightpage]{preview}

\usepackage{beton}
\usepackage{stmaryrd}

\usepackage[T1]{fontenc}
\usepackage[utf8x]{inputenx} 

\definecolor{purple}{rgb}{0.41, 0.16, 0.38}

\input{command}

\begin{document}

\title{Measures and slaloms}
\author[Piotr Borodulin--Nadzieja]{Piotr Borodulin-Nadzieja}
\address{Instytut Matematyczny, Uniwersytet Wroc\l awski}
\email{pborod@math.uni.wroc.pl}
\thanks{The first author was partially supported by Polish National Science Center grant 2013/11/B/ST1/03596 (2014-2017).}
\author[Tanmay Inamdar]{Tanmay Inamdar}
\address{School of Mathematics, University of East Anglia}
\email{t.inamdar@uea.ac.uk}

\date{7 April 2016}


\begin{abstract} 
	We examine measure-theoretic properties of spaces constructed using the technique of Todor\v{c}evi\'{c} from \cite[Theorem 8.4]{Todorcevic}. We show that the existence of strictly positive measures on such spaces depends on combinatorial properties of certain families of slaloms. As a corollary
	we get that if $\mathrm{add}(\mathcal{N}) = \mathrm{non}(\mathcal{M})$ then there is a non-separable space which supports a measure and which cannot be mapped continuously onto $[0,1]^{\omega_1}$. Also, without any additional axioms we prove that
	there is a non-separable growth of $\omega$ supporting a measure and that there is a compactification $L$ of $\omega$ with growth of such properties and such that the natural copy of $c_0$ is complemented in $C(L)$. Finally, we discuss examples of spaces not
	supporting measures but satisfying quite strong chain conditions.  Our main tool is a characterization due to Kamburelis (\cite{Kamburelis}) of Boolean algebras supporting measures in terms of their chain conditions in generic extensions by a measure algebra.\end{abstract}

\maketitle

\section{Introduction}

The study of the interplay between the countable chain condition and separability has been a constant source of many important results since the formulation of Suslin's Hypothesis. In general, it is hard to separate these properties assuming $\mathsf{MA}_{\omega_1}$ if we deal with compact spaces
which are in some sense topologically small. For example, under $\mathsf{MA}_{\omega_1}$ neither linearly ordered nor first-countable spaces can be ccc and non-separable. 

For many years the status of the following statement was unclear: \emph{every ccc compact space which cannot be mapped continuously onto $[0,1]^{\omega_1}$ is separable}. In
\cite[Remark 8.7]{Todorcevic}, this is referred to as `the ultimate form' of Suslin's Hypothesis. Quite
unexpectedly, in \cite{Todorcevic} Todor\v{c}evi\'{c} proved that it is inconsistent, i.e., there is a $\mathsf{ZFC}$ example of a ccc non-separable space which cannot be mapped continuously onto $[0,1]^{\omega_1}$. 

In \cite{Pbn-Plebanek} the authors consider a weakening of the above assertion: \emph{every compact space supporting a measure which cannot be mapped continuously onto $[0,1]^{\omega_1}$ is separable.} A space $K$ supports a measure if there is
measure $\mu$ on $K$ such that $\mu(U)>0$ for every nonempty open set $U\subseteq K$. This is clearly a stronger condition than ccc and still weaker than separability.  The results from \cite{Kunen-vanMill} imply that it does not hold if
$\mathrm{cov}(\mathcal{N}_{\omega_1})=\omega_1$ and \cite{Pbn-Plebanek} contains a counterexample under $\mathsf{MA}$. It is still not known if this statement is consistent with $\mathsf{ZFC}$.

In this context a natural question is whether the space from \cite{Todorcevic} mentioned above supports a measure. Consistently, it does not, for example if $\mathrm{add}(\mathcal{N}) = \omega_1 <\mathrm{cov}(\mathcal{N}_{\omega_1})$ (see Section \ref{meas}). However, examining 
Todor\v{c}evi\'{c}'s space more carefully, we found that its measure-theoretic properties depend on the way it is constructed (so it makes more sense to speak about \emph{Todor\v{c}evi\'{c} spaces}). We prove that under
$\mathrm{add}(\mathcal{N})=\mathrm{non}(\mathcal{M})$ we can modify Todor\v{c}evi\'{c}'s construction in
such a way that it supports a measure (Theorem \ref{small}).
It improves the result from \cite{Pbn-Plebanek} mentioned above.

Moreover, using similar techniques, we construct a $\mathsf{ZFC}$ example of a Boolean algebra supporting a measure which is not $\sigma$-centered and which can be embedded in $\mathcal{P}(\omega)/{\rm Fin}$ (Theorem \ref{growth}). Earlier, only consistent examples of such
spaces were known (see \cite{Drygier-Plebanek15}). Recall that if a Boolean algebra is $\sigma$-centered, then
it can be embedded in $P(\omega)/{\rm Fin}$, and by Parovi\v{c}enko's theorem (see \cite{Parovichenko}) the measure algebra (and in fact any Boolean algebra of the size of the continuum) can be embedded into $\mathcal{P}(\omega)/{\rm Fin}$ under $\mathsf{CH}$. 
On the other hand, a result of Dow and Hart (see \cite{Dow-Hart}) shows that under $\mathsf{OCA}$ the measure algebra cannot be embedded into $\mathcal{P}(\omega)/{\rm Fin}$, and in fact Selim has observed (see \cite{Selim}) that the same holds for any atomless Maharam algebra. Our result shows that a non-trivial piece of the measure algebra
can be embedded in $\mathcal{P}(\omega)/{\rm Fin}$ in $\mathsf{ZFC}$.

Recently, several authors (\cite{Castillo}, \cite{Drygier-Plebanek}) considered the problem of when $C(K)$, where $K$ is a compactification of $\omega$, contains a copy of $c_0$ which is complemented in $C(K)$. If the natural copy of $c_0$ is complemented in $C(K)$, then $K\setminus \omega$ supports a
measure, an observation which in \cite{Drygier-Plebanek} is attributed to Kubi{\' s}.
In \cite{Drygier-Plebanek} the authors prove that under $\mathsf{CH}$ there is a
compactification $K$ of $\omega$ such that $K\setminus \omega$ is non-separable, supports a measure, and the natural copy of $c_0$ in $C(K)$ is complemented. We prove that such a space exists in $\mathsf{ZFC}$ (Corollary \ref{c_0}).

Finally, we show in $\mathsf{ZFC}$ that Todor\v{c}evi\'{c}'s space can be constructed in such a way that it does not support a measure but satisfies quite strong chain conditions: $\sigma$-n-linkedness for every $n\in \omega$ and
Fremlin's property (*). Considering the completion of the Boolean algebra of clopen subsets of this space, we obtain a complete Boolean algebra which possesses these properties and which does not support a measure. This raises the natural, though
perhaps na\"{\i}ve, question of whether this algebra can provide another example of a non-measurable Maharam algebra. We show that forcing with this Boolean algebra adds a Cohen real and so it is not weakly distributive and thus cannot be a Maharam algebra.

Our main tool is Kamburelis' characterization of Boolean algebras supporting measures as Boolean algebras which can be made $\sigma$-centered by adding random reals (\cite{Kamburelis}). Thanks to this result and the nature of Todor\v{c}evi\'{c}'s construction, to check if
a Boolean algebra obtained in this way supports a measure it is enough to investigate destructibility of some families of slaloms by random forcing. Some of the theorems mentioned above can be proved directly by using Kelley's
characterization of Boolean algebras supporting measures (\cite{Kelley}). However, we decided to use the forcing language because this is how the results were obtained. Also, perhaps the facts concerning the destructibility of families of slaloms can be of independent interest.

\section{Notation and basic facts}

We use standard set theoretic notation. 
Let $\kappa$ be a cardinal number. Then by $\lambda_\kappa$ we denote the standard measure on $[0,1]^\kappa$ and by
$\mathbb{M}_\kappa$ the \emph{measure algebra of type $\kappa$}, that is, the Boolean algebra $\mathrm{Bor}[0,1]^\kappa/_{\lambda_\kappa=0}$. We write $\mathbb{M}$ instead of $\mathbb{M}_1$, which we simply call the \emph{measure algebra}.

By a \emph{real number} we will mean an element of Baire space, $\omega^\omega$, or an element of some $\prod_{n\in \omega} S_n$, where $S_n \subseteq \omega$, the exact choice of which shall be clear from the context. If $S \subseteq \omega \times \omega$, then $S(n)$ will denote the horizontal section $\{m \colon (n,m) \in S\}$.

Most of the spaces which will appear in this article will be constructed as Stone spaces of some Boolean algebras. 
We will treat Boolean algebras as algebras of sets, and so we will use ``$\cup$'' to denote conjunction, ``$\emptyset$'' to denote the zero element, ``$\subseteq$'' to denote the Boolean order, and so on. If $\mathfrak{A}$ is a Boolean algebra, then
$\mathfrak{A}^+ = \mathfrak{A}\setminus \{\emptyset\}$. A family $\mathcal{P}$ is a \emph{$\pi$-base} of a Boolean algebra $\mathfrak{A}$ if $\mathcal{P}\subseteq \mathfrak{A}^+$ and for each $A\in \mathfrak{A}^+$ there is $P\in \mathcal{P}$ such
that $P\subseteq A$. For a family $\mathcal{G}$ by ${\rm alg}(\mathcal{G})$ we denote the Boolean algebra generated by $\mathcal{G}$.

Recall that a Boolean algebra $\mathfrak{A}$ is $\sigma$-centered if $\mathfrak{A}^+=\bigcup_{n<\omega} \mathcal{C}_n$, where each $\mathcal{C}_n$ is centered (that is, each finite subset of $\mathcal{C}_n$ has non-empty intersection). A family $\mathcal{A}\subseteq \mathfrak{A}$ is
independent if for every collection of finite disjoint families $\mathcal{A}_0$, $\mathcal{A}_1 \subseteq \mathcal{A}$ we have
\[ \bigcap_{\mathcal{A}_0} A \cap \bigcap_{\mathcal{A}_1} A^c \ne 0. \]
A Boolean algebra is $\sigma$-centered if and only if its Stone space is separable and it contains an uncountable independent family if and only if the Stone space maps continuously onto $[0,1]^{\omega_1}$.

Recall that a Boolean algebra $\mathfrak{A}$ has the countable chain condition, abbreviated to `ccc', if any collection of pairwise disjoint elements from $\mathfrak{A}^+$ is at most countable.

If $\mathcal{I}$ is an ideal of subsets of $K$, then
\[ \mathrm{add}(\mathcal{I}) = \min\{|\mathcal{A}|\colon \mathcal{A}\subseteq \mathcal{I}, \ \bigcup \mathcal{A}\notin \mathcal{I}\}, \]
\[ \mathrm{non}(\mathcal{I}) = \min\{|X|\colon X\subseteq K, \ X\notin \mathcal{I}\}, \]
\[ \mathrm{cov}(\mathcal{I}) = \min\{|\mathcal{A}|\colon \mathcal{A}\subseteq \mathcal{I}, \ \bigcup \mathcal{A} = K\}. \]
By $\mathcal{N}$ we will mean the $\sigma$-ideal of Lebesgue null sets, by $\mathcal{M}$, the $\sigma$-ideal of meager sets, and by $\mathcal{N}_{\omega_1}$ the $\sigma$-ideal of $\lambda_{\omega_1}$-null sets. By ${\rm Fin}$ we will denote the ideal of finite subsets (of a set which should be clear from the context). We shall also need the standard fact that $\mathrm{add}(\mathcal{N})$ is an uncountable regular cardinal.

The \emph{bounding number} is defined by
\[ \mathfrak{b} = \min\{|\mathcal{F}| \colon \mathcal{F} \subseteq \omega^\omega, \ \forall g\in \omega^\omega \ \exists f\in \mathcal{F} \ f\nleq^* g\}. \]
Here $f \leq^* g$ means $f(n)\leq g(n)$ for all but finitely many $n$'s. Similarly, $A\subseteq^* B$ will denote the fact that $A\setminus B$ is finite. We shall need the standard fact that $\mathrm{add}(\mathcal{N}) \leq \mathfrak{b}$.

By a measure on a Boolean algebra we understand a \emph{finitely-additive} measure. Note that every such measure can be uniquely extended to a $\sigma$-additive Radon measure on the Stone space. 

Recall that a space $K$ has \emph{countable $\pi$-character} if each $x\in K$ has a local $\pi$-base (i.e., a family $\mathcal{U}_x$ of nonempty open sets such that each neighbourhood of $x$ contains an element of $\mathcal{U}_x$) which is countable. A space $K$
is \emph{scatteredly-fibered} if there is a continuous function $f\colon K\to M$, where $M$ is a metric space, such that each fiber $f^{-1}[x]$ is a scattered space (i.e. it cannot be mapped continuously onto $[0,1]^\omega$). Note that no scatteredly-fibered space
can be mapped continuously onto $[0,1]^{\omega_1}$. Otherwise, one of the fibers could be mapped continuously onto $[0,1]^{\omega_1}$ (by Tkachenko's theorem, see \cite{Tkachenko}). Similarly, one can define \emph{linearly-fibered} spaces. 

A \emph{compactification} of a space $X$ is a compact space $K\supseteq X$ such that $X$ is dense in $K$. The space $K\setminus X$ is called a \emph{growth} of $X$. We will consider compactifications of $\omega$ (with the discrete topology). If
$\mathfrak{A}$ is a subalgebra of $\mathcal{P}(\omega)$, then its Stone space is a compactification of $\omega$. Similarly, Stone spaces of subalgebras of $\pofin$ are growths of $\omega$. 

We are going to violate notation in many different ways. In particular, we will not always distinguish between Borel sets and the elements of $\mathbb{M}$. Also, we will not distinguish in notation between elements of Boolean algebras and clopen
subsets of its Stone space or between measures on Boolean algebras and its extensions to the Stone spaces. We hope this will not cause any confusion.

For proofs of the standard facts of Stone duality and Boolean algebras, see \cite{BAhandbook}; for set theory, see \cite{Jech}; for set theory of the reals, see \cite{Bartoszynski}; for Banach space theory, see \cite{Kalton06}.

\section{Todor\v{c}evi\'{c}'s construction}

In this short section we will explain some details of the construction from \cite[Theorem 8.4]{Todorcevic}. 

For $g \in \omega^ \omega$ let $\mathcal{S}_g$ be the set of $g$-\emph{slaloms}, i.e.,
\[ \mathcal{S}_g = \{S\subseteq \omega\times\omega \colon |S(n)|< g(n)\}. \]
Let $h\in \omega^\omega$ be given by $h(n)=2^n$. We write $\mathcal{S}$ for $\mathcal{S}_h$ (any increasing function $g$ such that $\sum_n \frac{1}{g(n)}$ converges would be equally good). Similarly, \emph{a slalom} will mean an $h$-slalom.

Let $\Omega = \{(S,n)\colon n\in\omega, \ S\in \mathcal{S}, \ S\subseteq (n\times 2^n)\}$.
For each $A\subseteq \omega\times\omega$ define 
\[ T_A = \{(T,n)\in \Omega \colon A\cap (n\times 2^n) \subseteq T\}. \]
For $(S,n)\in \Omega$ let
\[ T_{(S,n)} = \{(T,m)\in \Omega \colon m\geq n, T\cap (n\times 2^n) = S\}. \]

It will be convenient to collect some simple observations concerning $T_A$'s. 

\begin{fact} For each $A$, $B\in \mathcal{S}$ we have
	\begin{enumerate}
		\item $S\in \mathcal{S}$ if and only if $T_S$ is infinite,
		\item $T_{(A\cup B)} = T_A \cap T_B$,
		\item if $A\subseteq B$, then $T_B \subseteq T_A$,
	\end{enumerate}
\end{fact}

If $\mathcal{A}\subseteq \mathcal{P}(\omega\times \omega)$, then let $\mathfrak{T}_\mathcal{A}$ be the subalgebra of $\mathcal{P}(\Omega)$ generated by 
\[\{T_A \colon A\in \mathcal{A}\} \cup \{T_{(S,n)}\colon (S,n)\in\Omega\}.\]

Finally, let $K_\mathcal{A}$ be the Stone space of $\mathfrak{T}_\mathcal{A}/{\rm Fin}$. 

We say that a family $\mathcal{F}\subseteq \omega^\omega$ is \emph{localized} by $\mathcal{S}_g$ if there is $S\in \mathcal{S}_g$ such that $f\subseteq^* S$ (that is, for all but finitely many $n \in\omega$ we have that $f(n) \in S(n)$) for every
$f\in \mathcal{F}$. Similarly, a family $\mathcal{A} \subseteq \mathcal{S}_g$ is \emph{$\subseteq^*$-bounded}, or simply, \emph{bounded}, if there is a $S \in \mathcal{S}_g$ such that $A \subseteq^* S$ (that is, for all but finitely many $n \in\omega$ we have that $A(n) \subseteq S(n)$) for every $A \in \mathcal{A}$.

\begin{thm}\cite[Theorem 2.3.9]{Bartoszynski}  \label{Bartoszynski}
	Let $g\in \omega^\omega$ be such that $\lim_n g(n)=\infty$. Then
\[ \mathrm{add}(\mathcal{N}) = \min\{|\mathcal{F}|\colon \mathcal{F}\subseteq \omega^\omega, \mathcal{F}\mbox{ is not localized by }\mathcal{S}_g\}. \]
\end{thm}

Let 
\[ \mathcal{Z} = \{S\subseteq \omega\times \omega\colon S\in \mathcal{S}\mbox{ and } \lim_n \frac{1}{2^n}|S(n)| = 0 \}. \]

In \cite{Kunen-Fremlin} a subfamily of $\omega^\omega$ which cannot be localized by $\mathcal{S}$ was used to construct a family of elements of $\mathcal{Z}$ which is not $\subseteq^*$-bounded in $\mathcal{S}$. Note that in \cite{Kunen-Fremlin} and in
\cite{Todorcevic} the authors considered $\mathcal{S}_g$ for $g(n)=n$ instead of $\mathcal{S}$ but it does not make any difference for their results.

\begin{thm}\cite[Theorem 4]{Kunen-Fremlin}\label{kunen-fremlin} 
	There is a $\subseteq^*$-chain $\{A_\alpha\colon \alpha<\mathrm{add}(\mathcal{N})\} \subseteq \mathcal{Z}$ such that for every $S\in\mathcal{S}$ there is $\alpha<\mathrm{add}(\mathcal{N})$ such that $A_\alpha \nsubseteq^* S$.
\end{thm}

Let $\{A_\alpha\colon \alpha<\mathrm{add}(\mathcal{N})\}$ be a family given by Theorem \ref{kunen-fremlin}. 
Denote \[ \mathcal{A} = \{A\in \mathcal{S}\colon A=^* A_\alpha\mbox{ for some }\alpha<\mathrm{add}(\mathcal{N})\}.\]  
\begin{thm}\label{todor} \cite[Theorem 8.4]{Todorcevic} $K_\mathcal{A}$ has the following properties:
	\begin{enumerate}
		\item it is homeomorphic to a growth of $\omega$,
		\item it is non-separable, 
		\item it is ccc,
		\item it is linearly-fibered and scatteredly-fibered,
		\item it has countable $\pi$-character.
	\end{enumerate}	
	\end{thm}

\begin{proof} 
	For the proof see \cite[Theorem 8.4]{Todorcevic}. We will only present a slightly different proof that $\mathfrak{T}_\mathcal{A}/{\rm Fin}$ is not $\sigma$-centered (i.e. that $K_\mathcal{A}$ is not separable).

	Suppose for the contradiction that $\mathfrak{T}_\mathcal{A}/{\rm Fin} = \bigcup_{n<\omega} \mathcal{C}_n$ and each $\mathcal{C}_n$ is centered. Then, since $\mathrm{add}(\mathcal{N})$ is a regular uncountable cardinal, there is an $n$ such that \[\{\alpha\colon \exists A\in \mathcal{C}_n, A=^* A_\alpha\}\mbox{ is cofinal in }\mathrm{add}(\mathcal{N}).\] For
	simplicity we will just assume that \[ \{T_{A_\alpha}\colon \alpha<\mathrm{add}(\mathcal{N})\}/{\rm Fin} \subseteq \mathcal{C}_n.\]  
	Of course $\bigcup_\alpha A_\alpha \notin \mathcal{S}$ and so there is $m\in \omega$ such that $|\bigcup_\alpha A_\alpha(m)|\geq 2^m$. Enumerate $\bigcup_\alpha A_\alpha(m) = \{k_0, k_1, \dots\}$ and for each $i$ let $\alpha_i$ be such that $k_i \in
	A_{\alpha_i}(m)$. Then 
	\[ \bigcap_{i\leq 2^m} T_{A_{\alpha_i}} = \{(T,n) \in \Omega \colon (\bigcup_{i \leq 2^m}A_{\alpha_i}) \cap (n \times 2^n) \subseteq T\} \]
	does not contain any $(T,n)$ such that $n >m$, and hence is finite, a contradiction
\end{proof}

It will be convenient to make the following observation available.

\begin{remark}\label{T^*}
	Denote by $\mathfrak{T}^*_\mathcal{A}$ the Boolean subalgebra of $\mathfrak{T}_\mathcal{A}$ generated only by $\{T_A\colon A\in \mathcal{A}\}$. The above proof shows that $\mathfrak{T}^*_\mathcal{A}/{\rm Fin}$ is not $\sigma$-centered.
\end{remark}

\section{Random destructible families of slaloms}\label{destruct}

The main ingredient of the construction from Theorem \ref{todor} is a family of slaloms. In this section we will investigate combinatorial properties of certain families of slaloms which in Section \ref{applications} will be translated to properties
of resulting spaces.

Let \[ \mathcal{I} = \{S\subseteq \omega \times \omega\colon S(n)\subseteq 2^n \mbox{ for each }n\mbox{ and } \sum_n \frac{1}{2^n}|S(n)| < \infty \}. \]
Notice that if $f\colon \{(n,i)\colon i< 2^n, n\in\omega\} \to \omega$ is the natural enumeration function (sending $\{n\}\times 2^n$ to $[2^n, 2^{n+1})$ for each $n$), then $I\in \mathcal{I}$ if and only if $f[I]\in \mathcal{I}_{1/n}$, where
\[ \mathcal{I}_{1/n} = \{A\subseteq \omega\colon \sum_n \frac{|A\cap [2^n, 2^{n+1})|}{2^n} < \infty\}, \]
i.e., $\mathcal{I}_{1/n}$ is the classical summable ideal on $\omega$ (see for example \cite{Farah}).
Let 
\[ \mathcal{W} = \mathcal{I} \cap \mathcal{S}. \] 
In other words, $\mathcal{W}$ consists of elements of $\mathcal{I}$ which miss at least one point of $\{n\}\times 2^n$ for each $n$.

For $g\in \omega^\omega$ equip the space $\mathcal{X}_g = \prod g(n)$ with the product topology (so that $\mathcal{X}_g$ is homeomorphic to the Cantor set). Let $\lambda$ be the standard measure on $\mathcal{X}_g$, so in particular\[ \lambda(\{f\in
\mathcal{X}_g\colon f(n)=i\}) = \frac{1}{g(n)}\] if $i< g(n)$. 
Recall that $h \in \omega^\omega$ is given by $h(n)= 2^n$. Let $\mathcal{X} = \mathcal{X}_h$ and let $A_g = \{f\in \mathcal{X}\colon \exists^\infty n \ f(n)=g(n)\}$. 

Before we proceed, we make the simple observation that if $f \in \mathcal{X}$, then \[ \{(n,f(n))\colon  n\in \omega, n >0\}\in \mathcal{W},\] and also, if $S \in \mathcal{S}$ is such that $S(n) \subseteq 2^n$ for every $n$, then there is a $f \in \mathcal{X}$ and a $T \in \mathcal{S}$ such that $S \subseteq T$ and $T(n)= 2^n \setminus \{f(n)\}$ for every $n$. We shall use these observations several times in what follows.

\begin{prop}[folklore] \label{non}
	There is a family $\mathcal{F}\subseteq \mathcal{X}$ of size $\mathrm{non}(\mathcal{M})$ which is not localized by $\mathcal{S}$. 
\end{prop}

\begin{proof}
	First, notice that for each $g\in \mathcal{X}$ the set $A_g$ is comeager. Indeed, let $A^n_g = \{f\in \mathcal{X}\colon f(n)=g(n)\}$ for $g\in \mathcal{X}$. Of course, each $A^n_g$ is open and $\bigcup_{n>m} A^n_g$ is dense for each $m\in \omega$. But
	\[ A_g = \bigcap_m \bigcup_{n>m} A^n_g. \]
Let $\{f_\alpha\colon \alpha<\mathrm{non}(\mathcal{M})\} \subseteq \mathcal{X}$ be a family witnessing $\mathrm{non}(\mathcal{M})$. Then for each $g\in \mathcal{X}$ there is an $\alpha$ such that $f_\alpha \in A_g$ and so $f_\alpha(n)=g(n)$ for
infinitely many $n$.

The family $\{f_\alpha\colon \alpha<\mathrm{add}(\mathcal{N})\}$ is not localized by $\mathcal{S}$, because for every $S\in \mathcal{S}$ there is $g_S\in \mathcal{X}$ such that $g_S(n)\notin S(n)$ for each $n$. Hence, there is an $\alpha$ such that $f_\alpha(n)
= g_S(n)$ for infinitely many $n$. So, for each $S\in \mathcal{S}$ there is an $\alpha<\mathrm{non}(\mathcal{M})$ such that $\{n\colon f_\alpha(n)\notin S(n)\}$ is infinite.
\end{proof}

Now, as in \cite[Theorem 4]{Kunen-Fremlin}, we will use a set of reals as above to find a $\subseteq^*$-chain in $\mathcal{W}$ which is not $\subseteq^*$-bounded in $\mathcal{S}$. The proof is essentially the same as there, with some minor
modifications, but we include it here for the sake of completeness.

\begin{thm}\label{kunen-fremlin2} Assume $\mathrm{add}(\mathcal{N}) = \mathrm{non}(\mathcal{M})$. There is a $\subseteq^*$-chain $\{A_\alpha\colon \alpha<\mathrm{add}(\mathcal{N})\} \subseteq \mathcal{W}$ such that for every
	$S\in\mathcal{S}$ there
	is $\alpha<\mathrm{add}(\mathcal{N})$ such that $A_\alpha \nsubseteq^* S$. 
\end{thm}

\begin{proof}
	Let $\mathcal{F} = \{f_\alpha\colon \alpha<\mathrm{add}(\mathcal{N}\}$ be as in Proposition \ref{non}.

	Let $A_0 = f_0\cap([1,\infty)\times \omega)$ and assume that we have constructed $A_\alpha$'s for $\alpha<\beta$. 
	For each $\alpha<\beta$ fix a function $g_\alpha \colon \omega\to \omega$ such that 
	\[ \sum_{i\geq g_\alpha(n)} \frac{1}{2^i}|A(i)| < 1/2^n. \]
	As $\beta<\mathrm{add}(\mathcal{N})\leq \mathfrak{b}$, there is a function $g\colon \omega\to \omega$ which is strictly increasing and which $\leq^*$-dominates
	$\{g_\alpha \colon \alpha<\beta\}$. For each $\alpha< \beta$, fix $m_\alpha$ such that $g(n)\geq g_\alpha(n)$ for each $n \geq m_\alpha$.

	Define $F_\alpha\colon \omega \to [\omega\times \omega]^{<\omega}$ such that
$$
F_\alpha(n) = 
\begin{cases}
A_\alpha \cap [g(n), g(n+1))\times \omega \mbox{ if }n\geq m_\alpha, \\
 \emptyset \mbox{ otherwise.}
\end{cases}
$$
Now, since $[\omega\times\omega]^{<\omega}$ is countable and $\beta<\mathrm{add}(\mathcal{N})$, by Theorem \ref{Bartoszynski} applied to the space $\omega^{[\omega\times\omega]^{<\omega}}$ we see that there is an $f$-slalom $\Phi \subseteq \omega \times [\omega\times\omega]^{<\omega}$ for $f\in \omega^\omega$ given by $f(n) =n+1$ which localises all of the $F_\alpha$. That is,
\begin{enumerate}
		\item $\{n\colon F_\alpha(n) \notin \Phi(n)\}$ is finite,
		\item $|\{I\colon (n,I)\in \Phi\}|\leq n$,
\end{enumerate}
Additionally, throwing out some elements if needed, we can assume that
	\begin{enumerate}
		\item[(3)] for each $(n,I)\in \Phi$ there is $\alpha<\beta$ such that $F_\alpha(n) = I$.
	\end{enumerate}
	The last condition implies that whenever $(n,I)\in \Phi$, then $I\subseteq [g(n),g(n+1))\times \omega$ and $\sum_{i\geq g(n)} \frac{1}{2^i}|I(i)| < \frac{1}{2^n}$. Also, if $I$ is such that $(n,I)\in \Phi$ and $(k,l)\in I$, then $(k,l)\in A_\alpha$ for
	some $\alpha<\beta$ and as we will see at the end of the proof, therefore a $\gamma \leq \alpha$ such that $l= f_\gamma(k)$. 
	Let \[ A = \bigcup\{I\colon \exists n \ (n,I)\in \Phi\}. \] Notice that \[ \sum_{g(n)\leq i < g(n+1)} \frac{1}{2^i}|A(i)| < \frac{n}{2^n} \] and since $\sum_n \frac{n}{2^n} = 2$, we have that $A\in \mathcal{W}$. Moreover, for each $\alpha<\beta$ there is $m\geq m_\alpha$ such that $(n,F_\alpha(n))\in \Phi$ for every $n\geq m$. So, $A_\alpha
\subseteq A\cup [0,g(m)]\times \omega$ and it follows that $A_\alpha \subseteq^* A$.

Now, it is easy to see that there is a $k< \omega$ such that $(A\cup f_\beta) \cap ([k,\infty)\times \omega) \in \mathcal{W}$. Put
	\[ A_\beta = (A\cup f_\beta) \cap ([k,\infty)\times \omega). \]
	We have now finished the construction. To see that $\{A_\alpha\colon \alpha<\mathrm{add}(\mathcal{N})\}$ is not $\subseteq^*$-bounded by any slalom in $\mathcal{S}$, notice that the family $\mathcal{F}$ was chosen so as to not be localised by any slalom in $\mathcal{S}$, and since every real from this family is $\subseteq^*$-contained in some $A_\alpha$ (to be more specific, we simply have that if $\alpha< \mathrm{add}(\mathcal{N})$ then $f_\alpha \subseteq^* A_\alpha$), so the former family also inherits this property.

	Clearly \[\bigcup_{\alpha<\mathrm{add}(\mathcal{N})} A_\alpha \subseteq
	\bigcup_{\alpha<\mathrm{add}(\mathcal{N})} f_\alpha,\]
	so $A_\alpha(n) \subseteq 2^n$ for each $n$ and $\alpha<\mathrm{add}(\mathcal{N})$.

\end{proof}

\begin{remark}\label{non-cov}
Let 
\[ \kappa = \min\{|\mathcal{D}|\colon \mathcal{D}\subseteq \mathcal{W}, \neg \exists S\in \mathcal{S} \ \forall D\in \mathcal{D} \ D\subseteq^* S\}. \]
The reader may notice that Proposition~\ref{non} amounts to a proof that $\kappa \leq \mathrm{non}(\mathcal{M})$, and that Theorem~\ref{kunen-fremlin2} can actually be proved from the assumption that $\mathrm{add}(\mathcal{N}) = \kappa$.

In fact, there is a better upper bound for $\kappa$ (than $\mathrm{non}(\mathcal{M})$). Recall that if $\mathcal{I}$ is an ideal on $\omega$, then $\mathrm{cov}^*(\mathcal{I})$ is the minimal size of a subfamily of $\mathcal{I}$ such that for every infinite $X\subseteq \omega$ there is an element of the family intersecting $X$ on an infinite set (see
e.g. \cite{Hrusak07}). In this setting $\kappa$
is the minimal size of a subfamily of $\mathcal{I}_{1/n}$ such that for each subset of $\omega$ intersecting each interval $[2^n,2^{n+1})$ at least once, there is an element of this family intersecting it infinitely many times. Clearly then, $\kappa\leq \mathrm{cov}^*(\mathcal{I}_{1/n})$.

It is also not hard to see that $\mathrm{cov}(\mathcal{N}) \leq \kappa$. Indeed, for $W\in \mathcal{W}$ let $A_W =  \{f\in \mathcal{X}\colon \exists^\infty n \ f(n)\in W(n)\}$. By the Borel-Cantelli Lemma,
$\lambda(A_W)=0$ for each $W\in \mathcal{W}$. If $\mathcal{F}\subseteq \mathcal{W}$ is not bounded by any slalom, then each $f\in \mathcal{X}$ is in some $A_F$, $F\in \mathcal{F}$ (since we can in particular consider the slalom which on every $n$ is exactly $2^n\setminus
\{f(n)\}$). Hence, if $\mathcal{F}$ witnesses $\kappa$, then the family $\{A_F\colon F\in \mathcal{F}\}$
covers $\mathcal{X}$, and hence has size at least $\mathrm{cov}(\mathcal{N})$.

In fact, if $\mathrm{cov}(\mathcal{N})<\mathfrak{b}$, then $\kappa = \mathrm{cov}(\mathcal{N})$ (see \cite[Theorem 2.2]{Bartoszynski88}). It seems likely that consistently $\mathrm{cov}(\mathcal{N})<\kappa$, 
but we were not able to prove it.
\end{remark}

Assume that $\mathcal{F}\subseteq \omega^\omega$ is not localized by $\mathcal{S}_g$.
We say that $\mathcal{F}$ is $g$-\emph{destructible} by a forcing poset $\mathbb{P}$  if 
\[ \Vdash_\mathbb{P} ``\check{\mathcal{F}} \mbox{ is localized by }\dot{\mathcal{S}_g}". \]
Similarly, if $\mathcal{A}$ is a family not bounded by any element of $\mathcal{S}_g$, then we say that it is $g$-\emph{destructible} by $\mathbb{P}$ if 
\[ \Vdash_\mathbb{P} ``\check{\mathcal{A}} \mbox{ is bounded by }\dot{\mathcal{S}_g}". \]
As before \emph{destructible} means $h$-destructible ($h(n)=2^n$).

We will fix some notation for the rest of the article. 
For $n>0$, $k< 2^n$ let $I^n_k = \{f\in\mathcal{X}\colon f(n)=k\}$. We will consider the measure algebra $\mathbb{M}$ in the following incarnation: $\mathbb{M} = \mathrm{Bor}(\mathcal{X})/_{\lambda=0}$. Define an $\mathbb{M}$-name $\dot{S}$ for a subset of $\omega\times \omega$ in the following way:
\[ \llbracket k\in \dot{S}(n) \rrbracket = \mathcal{X}\setminus I^n_k. \]
Clearly, $\dot{S}$ is an $\mathbb{M}$-name for a slalom. In fact, \[ \Vdash_\mathbb{M} ``\exists \dot{f}\in \dot{\mathcal{X}} \ \dot{S}(n) = 2^n \setminus \{\dot{f}(n)\}", \]
and $\dot{f}$ is a name for a random real.

We will prove that the family $\mathcal{W}$ is destructible by $\mathbb{M}$. 

\begin{prop} \label{E-destructible} 
For every $W\in \mathcal{W}$ 
\[ \Vdash_\mathbb{M} ``\check{W} \subseteq^* \dot{S}" \]
\end{prop}
\begin{proof}
	Fix a $W\in \mathcal{W}$ and a $p\in
	\mathbb{M}$ of positive measure, and let $\varepsilon>0$ be such that $\lambda(p)>\varepsilon$. Take $n$ such that $\sum_{i>n} \frac{1}{2^i} |W(i)| < \varepsilon$. Clearly, 
	\[ \sum_{i>n} \lambda( \bigcup_{k \in W(i)} I^i_k ) <\varepsilon \]
and so if 
\[ q = \bigcup_{i>n} \bigcup_{k\in W(i)}  I^i_k, \]
then $\lambda(q)<\varepsilon$. So we finish by noticing that
\[ \emptyset \ne p \setminus q \Vdash ``\forall i>n \ \check{W}(i) \in \dot{S}(i)". \]
\end{proof}

\section{Applications}\label{applications}

\subsection{Non-separable growths of $\omega$ supporting a measure}\label{meas}

First, we are going to apply the results from the previous section to construct some non-separable ccc compact spaces. We will use the following theorem due to Kamburelis:

\begin{thm} \cite[Proposition 3.7]{Kamburelis}\label{kamburelis}
	A Boolean algebra $\mathfrak{A}$ supports a measure if and only if there is a cardinal $\kappa$ such that $\Vdash_{\mathbb{M}_\kappa}$ ``$\check{\mathfrak{A}}$ is $\sigma$-centered''.
\end{thm}

We will need the following fact.

\begin{prop}\label{pi-base} Assume that $\mathcal{B} \subseteq \mathcal{S}$ is closed under finite unions (as long as they belong to $\mathcal{S}$). Then the family $\{T_B\cap T_{(T,n)}\colon B\in \mathcal{B}, (T,n)\in \Omega\}/{\rm Fin} \setminus
	\{\emptyset\}$ forms a $\pi$-base of $\mathfrak{T}_\mathcal{B}/Fin$.
\end{prop}
\begin{proof}
	See Claim 1 of \cite[Theorem 8.4]{Todorcevic}.
\end{proof}

\begin{thm} \label{main}
	If $\mathcal{B}$ is closed under finite unions (as long as they belong to $\mathcal{S}$), and $\mathcal{B}$ is destructible by some $\mathbb{M}_\kappa$, then the Boolean algebra $\mathfrak{T}_\mathcal{B}/{\rm Fin}$ supports a measure. 
\end{thm}
\begin{proof}
	Assume that $\mathcal{B}$ is destructible by $\mathbb{M}_\kappa$. Let $V$ be the ground model and let $G$ be $V$-generic for $\mathbb{M}_\kappa$. We shall show that $\mathcal{B}$ is $\sigma$-centered in $V[G]$, which, by Theorem~\ref{kamburelis} will let us finish. But in fact we can get away with even less, since if we can show that some $\pi$-base of $\mathcal{B}$ is $\sigma$-centered, then from a countable partition of this $\pi$-base into centered sets we can easily get a countable partition of the whole of $\mathcal{B}^+$ into centered sets. And naturally the $\pi$-base that we have in mind is the one that is furnished to us by Proposition~\ref{pi-base}.
	
	We work in $V[G]$. We know that there is some $S\in \mathcal{S}$ (note that here we are applying the \emph{formula} for $\mathcal{S}$, so this might be strictly larger than $\mathcal{S}^V$) such that for every $B \in \mathcal{B}$ (here on the other hand we are considering $\mathcal{B}$ as a \emph{set} from the ground model) we have $B \subseteq^* S$.
	
	Let $\mathcal{D} = \{D\in \mathcal{S}\colon S=^*D\}$. 
	For $D\in \mathcal{D}$ and $(T,n)\in \Omega$ such that $T_D \cap T_{(T,n)}$ is infinite, let 
	\[ \mathcal{C}^D_{(T,n)} = \{T_B \cap T_{(T,n)}\in \mathfrak{T}_{\mathcal{B}}\colon B \in \mathcal{B}, B\subseteq D\},\]
	where we point out that if $B \subseteq D$, then $T_B \supseteq T_D$, so if $T_D \cap T_{(T,n)}$ is infinite, then since this set is contained in $T_B\cap T_{(T,n)}$, the latter set is infinite too.
	
	Clearly, for each such $D$ and $(T,n)$, the family $\mathcal{C}^D_{(T,n)}/{\rm Fin}$ is centered, and since we have that for every $B \in \mathcal{B}$ there is a $D \in \mathcal{D}$ such that $B \subseteq D$, we see that the collection of such $\mathcal{C}^D_{(T,n)}/{\rm Fin}$ is a countable covering of the $\pi$-base given to us by Proposition~\ref{pi-base} into centered sets, so we are done.
\end{proof}

\begin{thm}\label{small} Assume $\mathrm{add}(\mathcal{N})=\mathrm{non}(\mathcal{M})$. There is a non-separable growth of $\omega$ supporting a measure, which has countable $\pi$-character, and which is scatteredly- and linearly-fibered. In particular, it does not map continuously onto $[0,1]^{\omega_1}$. 
\end{thm}

\begin{proof}
	Let $\mathcal{A}$ be the closure under finite modifications (as long as they belong to $\mathcal{S}$) of a family $(A_\alpha)_{\alpha<\mathrm{add}(\mathcal{N})}\subseteq \mathcal{W}$ as given to us by Theorem \ref{kunen-fremlin2}. It is easy to see that since $\mathcal{A}$ is obtained from a $\subseteq^*$-chain by taking finite modifications (so long as they belong to $\mathcal{S}$), we have that $\mathcal{A}$ is closed under taking finite unions (so long as they belong to $\mathcal{S}$). By Proposition \ref{E-destructible} the family $\mathcal{A}$ is
	destructible by $\mathbb{M}$, so by Theorem \ref{main} the space $K_\mathcal{A}$ supports a measure.

	That $K_\mathcal{A}$ satisfies the other properties listed in
	the statement can be seen in the same way as in Theorem \ref{todor} (i.e., as in \cite[Theorem 8.4]{Todorcevic}), using in particular the fact that $(A_\alpha)_{\alpha<\mathrm{add}(\mathcal{N})}$ is not bounded by any element of $\mathcal{S}$.
\end{proof}

The axiom
$\mathrm{add}(\mathcal{N})=\mathrm{non}(\mathcal{M})$ above can be replaced by a milder assumption, see Remark \ref{non-cov}. 
Note that if $\mathrm{add}(\mathcal{N}) = \omega_1$ and $\mathrm{cov}(\mathcal{N}_{\omega_1})>\omega_1$ (or more generally, $\mathrm{add}(\mathcal{N}) = \kappa < \mathrm{cov}(\mathcal{N}_\kappa)$), then
a space like in the above theorem cannot be constructed using Theorem \ref{kunen-fremlin2}. The reason is that
$\mathrm{cov}(\mathcal{N}_{\omega_1})>\omega_1$ implies that all Boolean algebras of size $\omega_1$ and supporting a measure are $\sigma$-centered (see \cite[Lemma 3.6]{Kamburelis}). 
We do not know the answer to the following question.

\begin{prob}
	Is it consistent that there is no non-separable space supporting a measure which cannot be mapped continuously onto $[0,1]^{\omega_1}$?
\end{prob}

If the answer is negative, then perhaps one can construct a $\mathsf{ZFC}$ example using the techniques presented in this article for some appropriate family $\mathcal{F}\subseteq \mathcal{W}$ which is $\subseteq^*$-unbounded in $\mathcal{S}$. The most difficult part is to find a reason why
the resulting space does not map continuously onto $[0,1]^{\omega_1}$. In the construction of Todor\v{c}evi\'{c} (and our take on it), the $\subseteq^*$-linearity of $\mathcal{A}$ gives the linearly fibered-ness of the spaces, which then satisfy this
property thanks to Tkachenko's Theorem. Also, note that if $\mathrm{non}(\mathcal{N})=\omega_1$ and $\mathrm{cov}(\mathcal{N}_{\omega_1})>\omega_1$, then such a space
cannot have countable $\pi$-character (see \cite[Theorem 5.5]{Pbn-Plebanek}). 

\begin{remark}\label{ms}
The space of Theorem \ref{small} satisfies a slightly stronger property than the lack of continuous mapping onto $[0,1]^{\omega_1}$. Namely, it only carries separable measures. Recall that a measure $\mu$ is \emph{separable} 
if there is a countable family $\mathcal{A}$ of measurable sets such that for every measurable $E$ 
\[ \inf\{\mu(E \triangle A)\colon A\in \mathcal{A}\} = 0. \]
Every space which can be mapped continuously
onto $[0,1]^{\omega_1}$ carries a non-separable measure and the reverse implication does not hold in general (see \cite{Fremlin}). The fact that the space of Theorem \ref{small} does not carry a non-separable measures follows directly
from \cite[Theorem 3.1]{Drygier} which says that scatteredly-fibered spaces only carry separable measures.
\end{remark}

Denote by $\mathcal{E}$ the $\sigma$-ideal on $2^\omega$ generated by the closed sets of measure zero. In \cite{Drygier-Plebanek} the authors proved that if $\mathrm{cf}(\mathrm{cov}(\mathcal{E})^\omega)< \mathfrak{b}$, then there is a non-separable growth of $\omega$ supporting a measure. Using the above technology we can prove that such a space exists in $\mathsf{ZFC}$. We can demand that the space, contrary to that of Theorem \ref{small}, is quite big in the combinatorial sense, that is, it maps continuously onto $[0,1]^\mathfrak{c}$.

\begin{thm}\label{growth} There is a non-separable growth $K$ of $\omega$ which supports a measure and which can be mapped continuously onto $[0,1]^\mathfrak{c}$.
\end{thm}

\begin{proof}
	Consider $\mathcal{W}$ as in the previous section. It is clear from the definition that $\mathcal{W}$ is closed under finite unions (so long as they belong to $\mathcal{S}$). As before, it follows that $K_\mathcal{W}$ supports a measure. We will show that it is not separable. Towards a contradiction, assume that it is, so in particular we get a countable covering $\{T_W\colon W\in \mathcal{W}\} = \bigcup_n \mathcal{C}_n$ where each $\mathcal{C}_n$ is centered.
	Therefore, for each $n$ the set $C_n = \bigcup \mathcal{C}_n \in \mathcal{S}$ (see the proof of Theorem \ref{todor}), and further, for every $W$ such that $T_W\in \mathcal{C}_n$, $W \subseteq C_n$. Let $f\in \mathcal{X}$ be such that $f(n)\not\in C_n(n)$ for each $n$. Clearly, $f\cap([1,\infty)\times \omega)\in \mathcal{W}$, but is not contained in $C_n$ for any $n$, a contradiction.

	We will finish the proof by showing that $\mathfrak{T}_\mathcal{W}/{\rm Fin}$ contains an independent family of size $\mathfrak{c}$, which clearly suffices.
	Let $\{X_\alpha \colon \alpha < \con\}$ be subsets of $\omega$ such that they are representatives of an independent family in $\mathcal{P}(\om)/{\rm Fin}$ (see \cite{Kantorovich}, \cite{Geschke}). Note that this in particular implies that all of them are infinite. Let $S \in \mathcal{W}$ be such that for every $n>1$, $|S(n)| \geq 2$ (and $S(0)=S(1)=\emptyset$). Also, for each $n>1$, let $Z^n_0, Z^n_1$ be non-empty pairwise disjoint subsets of $S(n)$. For each $\alpha < \con$, we shall define $S_\alpha \in \mathcal{W}$ as follows:
	
$$
S_\alpha(n) = 
\begin{cases}
Z^n_1 \mbox{ if }n\in X_\alpha, \\
Z^n_0 \mbox{ otherwise.}
\end{cases}
$$	

It is clear that each $S_\alpha$ is contained in $S$ (and hence is in $\mathcal{W}$) and is infinite. The following claim will then complete the proof of the theorem.

	\begin{claim} The family $\{ T_{S_\alpha} \colon \alpha < \con\}/{\rm Fin}$ is an independent family of $\mathfrak{T}_\mathcal{W}/{\rm Fin}$.
\end{claim}
\begin{why}
Let $\alpha_1, \ldots, \alpha_m, \beta_1,\ldots, \beta_n$ be pairwise distinct ordinals less than $\con$. We need to show that 
\[|\bigcap_{1\leq i \leq m}T_{S_{\alpha_i}}\cap \bigcap_{1\leq j\leq m}(T_{S_{\beta_j}})^c| =\azero.\]
Now, since $\{X_\alpha \colon \alpha < \con\}/{\rm Fin}$ is an independent family, we can find an infinite $Y\subseteq \om$ such that 
\[Y \subseteq \bigcap_{1\leq i \leq m}X_{\alpha_i}\cap \bigcap_{1\leq j\leq m}(X_{\beta_j})^c.\]
Then, let $T\subseteq \om \times \om$ be defined as follows: 

$$
T(n) = 
\begin{cases}
Z^n_1 \mbox{ if }n\in Y, \\
S(n) \mbox{ otherwise.}
\end{cases}
$$	

Notice that $T$ is infinite, and also, since $T \subseteq S$, that $T \in \mathcal{W}$. Also, for $1 \leq i \leq m$, $S_{\alpha_i} \subseteq T$, since the only $n < \om$ when $T(n) \ne S(n)$ are the $n \in Y$, in which case $T(n) = S^n_1 = S_{\alpha_i}(n)$ since $Y \subseteq X_{\alpha_i}$. But also, since $Y \subseteq (X_{\beta_j})^c$ for each $1 \leq j \leq n$, we have that for $n \in Y$, $T(n) \cap S_{\beta_j}(n) = \emptyset$, with the latter set being non-empty. 

It follows that if $l< \om$ is the least element of $Y$, then for every $k > l$, 
\[(T \cap (k\times 2^k), k) \in \bigcap_{1\leq i \leq m}T_{S_{\alpha_i}}\cap \bigcap_{1\leq j\leq m}(T_{S_{\beta_j}})^c,\]
thus yielding that the latter set is infinite and finishing the proof of the claim.
\end{why}
\end{proof}

\begin{remark}\label{T^*2} Actually, the above proof shows that the algebra $\mathfrak{T}^*_\mathcal{W}/{\rm Fin} = {\rm alg}(\{T_W\colon W\in \mathcal{W}\})/{\rm Fin}$ is not $\sigma$-centered (cf. Remark \ref{T^*}).
\end{remark}

\begin{remark}
After we carried out the above construction Tomasz \.{Z}uchowski presented another example of a Boolean algebra $\mathfrak{A}$ which supports a
	measure $\mu$, which  is not $\sigma$-centered and 
	such that there is an embedding $\varphi\colon \mathfrak{A} \to \mathcal{P}(\omega)/Fin$ (see \cite{Zuchowski}). His construction is quite different to ours and has the additional property that $\varphi$ transfers $\mu$ to the density $d$ (i.e., $\mu(A) = d(\varphi(A))$ for each $A\in
	\mathfrak{A}$). 
\end{remark}

\subsection{$c_0$-complementedness}

A closed subspace $Y$ of a Banach space $X$ is \emph{complemented in} $X$ if there is a projection $p\colon X\to X$ such that $p[X]=Y$. Since many Banach spaces have a copy of $c_0$ as a subspace, the question of which of these copies are
complemented is a natural one and was considered by many authors. One of the most important results in this topic is that of Sobczyk: if $X$ is separable, then each copy of $c_0$ in $X$ is complemented (\cite{Sobczyk}). On the other hand, a Banach
space $X$ is Grothendieck (i.e., $X^*$ does not contain weakly$^*$ convergent sequences which are not weakly convergent) if and only if no copy of $c_0$ in $X$ is complemented (see \cite{Cembranos84}). 

If $K$ is a compactification of $\omega$, then the space
\[ \{f\in C(K)\colon f(x)=0 \mbox{ for }x\in K\setminus \omega\} \]
forms a copy of $c_0$ in $C(K)$. We will call it the \emph{natural copy of} $c_0$ in $C(K)$. In \cite{Drygier-Plebanek} the authors discus when this natural copy is complemented in $C(K)$. If $K$ is metrizable, then $C(K)$ is separable and so, by
Sobczyk's theorem, every copy of $c_0$ in $C(K)$ is complemented. The following result implies that there are many compactifications $K$ of $\omega$ such that the natural copy of $c_0$ is not complemented in $C(K)$.

\begin{lem}\cite[Lemma 3.1]{Drygier-Plebanek}\label{drygier}
	Let $\mathfrak{A}$ be a subalgebra of $\mathcal{P}(\omega)$ containing all finite subsets and let $K$ be its Stone space. The following conditions are equivalent:
	\begin{enumerate}
		\item the natural copy of $c_0$ is complemented in $C(K)$;
		\item there is a sequence of measures $(\nu_n)_n$ on $\mathfrak{A}$ such that each $\nu_n$ vanishes on finite subsets and 
	\[ \lim_{n\to\infty} \nu_n(A) - \delta_n(A) = 0 \]
	for every $A\in\mathfrak{A}$
	\end{enumerate}
\end{lem}
Theorem \ref{drygier} implies, in particular, the following fact, which is attributed to Kubi{\'s} in \cite{Drygier-Plebanek}.

\begin{cor}\label{sp-measure}
If $K$ is a compactification of $\omega$ and the natural copy of $c_0$ is complemented in $C(K)$, then $K\setminus \omega$ supports a measure.
\end{cor}
\begin{proof} If $\mathfrak{A}\subseteq P(\omega)$
	is such that $K$ is its Stone space and $(\nu_n)_n$ is a sequence as in Theorem \ref{drygier}, then
$\nu$ given by
\[ \nu(A) = \sum_n \frac{1}{2^{n+1}}\nu_n(A), \]
for $A\in \mathfrak{A}$, is a measure on $\mathfrak{A}$ vanishing only on finite sets. Therefore, $\nu$ induces a strictly positive measure on $\mathfrak{A}/{\rm Fin}$ which can be extended further to a (strictly positive) measure on $K\setminus
\omega$. 
\end{proof}

In \cite[Theorem 5.1]{Drygier-Plebanek} the authors show that under $\mathsf{CH}$ there is a non-separable space $K$, a compactification of $\omega$ with a growth supporting a measure, such that the natural copy of $c_0$ in $C(K)$ is complemented. We will show that such a space exists in $\mathsf{ZFC}$. 

The space will be similar to $K_\mathcal{W}$ but this time the generators of the form $T_{(S,n)}$ will be a little bit cumbersome. So, we will consider the algebra $\mathfrak{T}^*_\mathcal{W}$, instead of $\mathfrak{T}_\mathcal{W}$. We make the
simple observation that this algebra contains every finite subset of $\Omega$. We also note that if $W \in \mathcal{W}$, then $T_W$ is infinite, and that if $W \not = W'$ are in $\mathcal{W}$, then $T_W \triangle T_{W'}$ is infinite.

Let $\dot{U}$ be an $\mathbb{M}$-name for a slalom. Define a function $f_{\dot{U}}\colon \{T_W\colon W\in \mathcal{W}\}/{\rm Fin} \to \mathbb{M}$ in the following way:
\[ f_{\dot{U}}([T_W]) = \llbracket \check{W} \subseteq \dot{U} \rrbracket. \]
Here $[T_W] = \{A\subseteq \Omega \colon A=^* T_W\}$. We will show that it can be extended to a homomorphism.

\begin{prop}\label{homo}
	For each $\mathbb{M}$-name $\dot{U}$ for a slalom the function $f_{\dot{U}}$ can be extended to a homomorphism $\varphi_{\dot{U}} \colon \mathfrak{T}^*_\mathcal{W}/{\rm Fin} \to \mathbb{M}$. 
\end{prop}

\begin{proof}
	Since $\{T_W \colon W\in \mathcal{W}\}/{\rm Fin}$ generates $\mathfrak{T}^*_\mathcal{W}/{\rm Fin}$, by Sikorski's theorem it is enough to check that if $A_0, \dots , A_k, B_0, \dots, B_l \in \mathcal{W}$ and \[ C =
	T_{A_0}\cap \dots \cap T_{A_k} \cap T_{B_0}^c
\cap \dots \cap T_{B_l}^c\] is finite, then \[ C' = f_{\dot{U}}([T_{A_0}]) \cap \dots \cap f_{\dot{U}}([T_{A_k}]) \cap (f_{\dot{U}}([T_{B_0}]))^c \cap \dots (f_{\dot{U}}([T_{B_l}]))^c = \emptyset.\]
	First, we will look at two particular cases:	
	\begin{enumerate}
		\item There is $n$ such that $|\bigcup_{i\leq k} A_i(n)|=2^n$. Then \[ C' \subseteq \bigcap_{i\leq k} \llbracket \check{A}_i(n) \subseteq \dot{U}(n)\rrbracket = \llbracket \bigcup_{i\leq k} \check{A}_i(n) \subseteq
			\dot{U}(n)\rrbracket = \emptyset,\]
			since $\dot{U}$ is a name for a slalom.
		\item There is $j\leq l$ such that $B_j \subseteq \bigcup_{i\leq k} A_i$. Then \[ C' \subseteq \llbracket \bigcup_{i\leq k} \check{A}_i \subseteq \dot{U}\rrbracket \cap \llbracket \check{B}_j \nsubseteq \dot{U}\rrbracket = \emptyset.\]
	\end{enumerate}
	Assume now that neither of the above is satisfied. In this case $(\bigcup_{i\leq k} A_i \cap (n\times 2^n), n)\in \Omega$ for each $n$ and 
	\[ C \supseteq \{(\bigcup_{i\leq k} A_i \cap (n\times 2^n), n)\colon n\in \omega\}. \]
	The latter set is clearly infinite and we are done.
\end{proof}

Thanks to Proposition \ref{homo} we can induce measures by names for slaloms. Note that usually these measures need not be positive on all infinite elements of $\mathfrak{T}^*_\mathcal{W}$.

\begin{cor}\label{measure} Let $\dot{U}$ be an $\mathbb{M}$-name for a slalom. The following formula uniquely defines a measure on $\mathfrak{T}^*_\mathcal{W}$:
\[ \nu(T_W) = \lambda(\llbracket \check{W}\subseteq \dot{U} \rrbracket). \]
This measure vanishes on finite sets.
\end{cor}

We are going to show that there is a sequence of measures defined on $\mathfrak{T}^*_{\mathcal{W}}$ as in Theorem \ref{drygier}.
We will use the name $\dot{S}$ constructed in the previous section. For $(T,m)\in \Omega$ define an $\mathbb{M}$-name $\dot{S}_{(T,m)}$ in the following way:
$$
\llbracket k\in \dot{S}_{(T,m)}(n) \rrbracket =
\begin{cases}
	\llbracket k\in \dot{S}(n)\rrbracket \mbox{ if } n\geq m \\
	\mathcal{X} \mbox{ if }  n< m \mbox{ and } k\in T(n) \\
	\emptyset \mbox{ if } n < m \mbox{ and } k\notin T(n).
\end{cases}
$$

Then $\dot{S}_{(T,m)}$ is a name for a slalom for each $(T,m)\in \Omega$ since $T$ is a slalom and $\dot{S}$ is a name for a slalom.

For $W\in \mathcal{W}$ and $(T,m)\in\Omega$, Corollary \ref{measure} allows us to define $\nu_{(T,m)}$ on $\mathfrak{T}^*_\mathcal{W}$ by setting
\[ \nu_{(T,m)}(T_W) = \lambda(\llbracket \check{W} \subseteq \dot{S}_{(T,m)} \rrbracket). \]
Of course, each $\nu_{(T,m)}$ vanishes on finite sets.

\begin{prop}\label{seq} For every $A\in \mathfrak{T}^*_\mathcal{W}$ 
	\[ \lim_{(T,m)\in\Omega} \nu_{(T,m)}(A)-\delta_{(T,m)}(A) = 0. \]
\end{prop}
\begin{proof}
	Let $W\in \mathcal{W}$.
\medskip

\begin{claim}$\lim_{(T,m)\in T_W}\nu_{(T,m)}(T_W) = 1$.\end{claim}

\begin{why}Let $\varepsilon>0$. There is $m$ such that $\lambda(\llbracket \forall n>m \ \check{W}(n)\subseteq \dot{S}(n)\rrbracket)>1-\varepsilon$. So, for each $n>m$ and $(T,n)\in T_W$
\[ \nu_{(T,n)} (T_W) = \lambda(\llbracket \check{W} \subseteq \dot{S}_{(T,n)}\rrbracket) \geq \lambda\left(\llbracket \check{W}\cap (n\times 2^n+1) \subseteq \check{T} \rrbracket \cap \llbracket \forall i \geq n \ \check{W}(i)\subseteq
\dot{S}(i)\rrbracket\right) = \]\[ = \lambda(\llbracket \forall i \geq n \ \check{W}(i)\subseteq
\dot{S}(i)\rrbracket)> 1-\varepsilon. \]
\end{why}

\begin{claim}$\lim_{(T,m)\notin T_W}\nu_{(T,m)}(T_W) = 0$.\end{claim}

\begin{why}In fact, if $(T,m)\notin T_W$, then
\[ \nu_{(T,m)}(T_W) = \lambda(\llbracket \check{W} \subseteq \dot{S}_{(T,m)}\rrbracket) \leq \lambda(\llbracket \check{W} \cap (m\times 2^m+1) \subseteq \check{T}\rrbracket) = 0. \]
\end{why}

In this way we have proved that $\lim_{(T,n)} \nu_{(T,n)}(T_W) - \delta_{(T,n)}(T_W) = 0$ for each $W \in \mathcal{W}$. Each element of $\mathfrak{T}^*_\mathcal{W}$ is a finite Boolean combination of elements of this form, so the convergence for arbitrary elements of the algebra easily follows. 
\end{proof}

\begin{cor} \label{c_0} There is a compactification $L$ of $\omega$ such that 
	\begin{enumerate}
		\item $L\setminus \omega$ is non-separable and supports a measure,
		\item the natural copy of $c_0$ is complemented in $C(L)$.
	\end{enumerate}
	If $\mathrm{add}(\mathcal{N})=\mathrm{non}(\mathcal{M})$, then we can additionally require that $L\setminus \omega$ does not map continuously onto $[0,1]^{\omega_1}$, and that every isomorphic copy of $c_0$ in $C(L)$ contains a further copy of $c_0$ which is complemented. That is, $C(L)$ is hereditarily separably Sobczyk.
\end{cor}
\begin{proof}
	Let $L$ be the Stone space of $\mathfrak{T}^*_\mathcal{W}$. Since $\mathfrak{T}^*_\mathcal{W}/{\rm Fin}$ is a subalgebra of $\mathfrak{T}_\mathcal{W}/{\rm Fin}$, it supports a measure.  Also, $\mathfrak{T}^*_\mathcal{W}/{\rm Fin}$ is not
	$\sigma$-centered (see Remark \ref{T^*2}). So, the space $L\setminus \omega$ supports a measure and is not separable. 
	Proposition \ref{seq} and Theorem \ref{drygier} imply (2). 
	
	If
	$\mathrm{add}(\mathcal{N})=\mathrm{non}(\mathcal{M})$, then instead of $\mathfrak{T}^*_\mathcal{W}$ we can use $\mathfrak{T}^*_\mathcal{A}$, where $\mathcal{A}$ is defined as in Theorem \ref{todor}. Again, since $\mathfrak{T}^*_\mathcal{A}$ is a
	subalgebra of $\mathfrak{T}_\mathcal{A}$, the algebra $\mathfrak{T}^*_\mathcal{A}/Fin$ supports a measure and does not contain an uncountable independent sequence. It is not  
	$\sigma$-centered by Remark
	\ref{T^*}. Also, since the Stone space of $\mathfrak{T}_\mathcal{A}$ only carries separable measures and $\mathfrak{T}^*_\mathcal{A}$ is a subalgebra of $\mathfrak{T}_\mathcal{A}$, the Stone space of $\mathfrak{T}^*_\mathcal{A}$ also does not carry a non-separable measure, and so by \cite[Theorem 8.4]{Drygier-Plebanek}, we have that $C(L)$ is hereditarily separably Sobczyk. \end{proof}

In the Section \ref{meas} we proved that the algebra $\mathfrak{T}_\mathcal{W}/{\rm Fin}$ supports a measure without giving any example of a supported measure. Theorem \ref{seq} allows
us to define strictly positive measure on $\mathfrak{T}^*_\mathcal{W}/{\rm Fin}$ explicitly (see the proof of Corollary \ref{sp-measure}). For each $(S,n)\in \Omega$ let $\mu_{(S,n)}$ be a measure on $\mathfrak{T}^*_\mathcal{W}/{\rm Fin}$ induced by
$\nu_{(S,n)}$. Fix an enumeration $\phi\colon \Omega \to \omega$.

\begin{cor}
The measure $\mu$ given by
\[ \mu(A) = \sum_{(S,n)\in \Omega} \frac{1}{{2^{\phi(S,n)+1}}} \mu_{(S,n)} (A) \]
is strictly positive on $\mathfrak{T}^*_\mathcal{W}/{\rm Fin}$.
\end{cor}

\subsection{Small $\sigma$-$n$-linked spaces supporting no measure.}

In \cite{Dzamonja-Plebanek} the authors claimed to prove that the space constructed by Todor\v{c}evi\'{c} does not support a measure. They were mainly interested in the corollary saying that there is a Boolean algebra of size
$\mathrm{add}(\mathcal{N})$ which does not support a measure. However, the argument in they proof was inaccurate 
(and the proof of Theorem \ref{small} indicates that the theorem is not true in general). Namely, it is
assumed that the Todor\v{c}evi\'{c} space is the Stone space of a Boolean algebra generated by a family of elements such that every two of them are either disjoint or ordered by inclusion. This is not the case, and in fact one can show that under Suslin's Hypothesis every
Boolean algebra with this property is $\sigma$-centered.

However, if we require some additional properties for the family $\mathcal{A}$ used for the construction, we can prevent the resulting space from supporting a measure (and, so the corollary of \cite[Theorem 3.1]{Dzamonja-Plebanek} remains true). 

The following is in principle \cite[Theorem 3.7]{Judah-Shelah}. Recall that $h \in \omega^\omega$ is given by $h(n)=2^n$.

\begin{prop}
	Let $g\in \omega^\omega$ be such that $\lim_n \frac{g(n)}{2^n} = \infty$. If $\mathcal{F}$ is not localized by $\mathcal{S}_g$, then it is not ($h$-)destructible by any $\mathbb{M}_\kappa$.  
\end{prop}

\begin{proof}
	Let $\dot{S}$ be a $\mathbb{M}_\kappa$-name for an ($h$-)slalom. Define $A\subseteq \omega\times\omega$ by
	\[ A(n) = \{k\colon \lambda_\kappa(\llbracket k\in \dot{S}(n) \rrbracket) > \frac{2^n}{g(n)}\}. \]
	Then $|A(n)|< \frac{g(n)}{2^n}\cdot 2^n = g(n)$. So, there is an $f\in\mathcal{F}$ such that $f\nsubseteq^* A$.
	Let $p\in \mathbb{M}_\kappa$ and $N\in \omega$ be such that 
	\[ p \Vdash ``\forall n>N \ \check{f}(n) \in \dot{S}(n)". \]
	There is $m>N$ such that $\lambda_\kappa(p)>\frac{2^m}{g(m)}$ and $f(m)\notin A(m)$. Then $\lambda_\kappa(\llbracket f(m) \in \dot{S}(n) \rrbracket)\leq \frac{2^m}{g(m)}$ and so
	\[ \emptyset \ne p \setminus \llbracket \check{f}_\alpha(m) \in \dot{S}(m) \rrbracket \Vdash ``\check{f}_\alpha(m) \notin \dot{S}(m)", \]
	a contradiction.
\end{proof}

It follows that we have, in $\mathsf{ZFC}$, families $\mathcal{F} \subseteq \omega^\omega$ of size $\mathrm{add}(\mathcal{N})$ which are not ($h$-)destructible by any $\mathbb{M}_\kappa$: simply consider a family $\mathcal{F}$ as in Theorem~\ref{Bartoszynski} which is not localised by $\mathcal{S}_g$ where $g \in \omega^\omega$ grows fast enough as above.

\begin{thm} \label{nomeasure}
There is a space $K$ satisfying the properties of the space of Theorem \ref{todor} and such that $K$ does not support a measure. Moreover, $K$ is $\sigma$-$n$-linked for every $n$.
\end{thm}

\begin{proof} Assume $\mathcal{F}$ is not $h$-destructible by any $\mathbb{M}_\kappa$. Then we can repeat the construction of Theorem \ref{kunen-fremlin} obtaining a $\subseteq^*$-increasing family $\mathcal{A}\subseteq \mathcal{Z}$ which is not
	destructible by any $\mathbb{M}_\kappa$. Then we can repeat the argument from the proof of Theorem \ref{todor} to show that $\mathfrak{T}_\mathcal{A}/{\rm Fin}$ (here and for the rest of this proof we are actually assuming for notational convenience that $\mathcal{A}$ is the closure under finite modifications of what it was in the last sentence) cannot be $\sigma$-centered in a forcing extension by $\mathbb{M}_\kappa$ for any $\kappa$. By Theorem \ref{kamburelis} the algebra $\mathfrak{T}_\mathcal{A}/{\rm Fin}$ does not support a measure.
	
We now show that $\mathfrak{T}_\mathcal{A}/{\rm Fin}$ is $\sigma$-$n$-linked for every $n$. As before, it suffices to show that the $\pi$-base for this algebra given to us by Proposition~\ref{pi-base} is $\sigma$-$n$-linked for every $n$ (note that we have implicitly used this proposition in this proof already, and it is easy to check that its antecedent is satisified by $\mathcal{A}$). In fact, it can easily be seen that this itself would follow if we can show for any $n$ that there is a countable covering $\mathcal{A} = \bigcup_m \mathcal{C}_m$ where for every $m$ and $A_1, A_2, \ldots A_n$ from $\mathcal{C}_m$, there is some $S \in \mathcal{S}$ such that $S \supseteq \bigcup \{A_1, A_2, \ldots A_n\}$.

But this is easily accomplished. Let $n$ be fixed. With each $A \in \mathcal{A}$ we associate a pair $(k_A, T_A)$ such that for every $m \geq k_A$, $\frac{1}{2^m}|A(m)| < \frac{1}{n}$, and $T_A = A \cap ([0, k_A) \times \omega)$. Since there are only countably many pairs and every $A \in \mathcal{A}$ can be associated with some such pair, it is clear that these pairs give rise to a countable partitioning of $\mathcal{A}$. Let $A_1, A_2, \ldots A_n$ be associated with the same pair, say $(k, T)$. Let $S = \bigcup \{A_1, A_2, \ldots A_n\}$. Now, it is clear that $S \in \mathcal{S}$: if $m < k$, then $S(m) = T(m)$, from which it follows that $|S(m)| < 2^m$; on the other hand, if $m \geq k$, then $S(m)= \bigcup \{A_1(m), A_2(m),\ldots A_n(m)\}$, so $|S(m)| \leq |A_1(m)|+|A_2(m)|+ \ldots |A_n(m)| < \frac{2^m}{n}\cdot n = 2^m$.
\end{proof}

We will show that using this technique we can find a space supporting no measure and satisfying Fremlin's property (*), a chain condition very close to separability. 

\begin{defi}[see \cite{Fremlin-Maharam}]
	A Boolean algebra $\mathfrak{A}$ has \emph{property (*)} if $\mathfrak{A}^+ = \bigcup \mathcal{C}_n$, where for each $n$ and for each infinite $\mathcal{C}\subseteq \mathcal{C}_n$, the family $\mathcal{C}$ can be furthermore refined to an infinite
	centered family.
\end{defi}

Define now the following family
\[ \mathcal{J} = \{S\subseteq \omega \times \omega\colon S(n)\subseteq 2^n \mbox{ for each }n\mbox{ and } \lim_n \frac{1}{2^n}|S(n)| = 0 \}. \]
Recall that the density zero ideal $\mathcal{D}$ (see for example \cite{Farah}) is defined by
\[ \mathcal{D} = \{A\subseteq \omega\colon \lim_n \frac{|A\cap [2^n, 2^{n+1})|}{2^n} = 0 \}. \]
Let 
\[ \mathcal{V} = \mathcal{J} \cap \mathcal{S}. \] 
Clearly, $\mathcal{V}$ is to the density zero ideal what $\mathcal{W}$ is to the summable ideal, and the same natural enumeration function witnesses this correspondence. 

\begin{thm}[Brendle-Yatabe, see \cite{Brendle-Yatabe} and \cite{Elekes}]
The density zero ideal is random-indestructible, i.e., 
\[ \Vdash_\mathbb{M} ``\mbox{There is no co-infinite }\dot{X}\mbox{ such that }\check{D}\subseteq^* \dot{X}\mbox{ for each }\check{D}\in \check{\mathcal{D}}." \]
\end{thm}

\begin{prop}\label{nomeasure2}
	The Boolean algebra $\mathfrak{T}_\mathcal{V}/Fin$ does not support a measure.
\end{prop}
\begin{proof}
First, notice that the density zero ideal is not destructible by $\mathbb{M}_\kappa$ for any cardinal $\kappa$. Indeed, for any new subset of $\omega$ in the generic extension, we can consider a nice name for this set, and because $\mathbb{M}_\kappa$ has the ccc, this name is decided by countably many conditions, so each new subset (potentially destroying the ideal) added by $\mathbb{M}_\kappa$ can be added by a single random real. See also \cite[Remark 3.4]{Elekes}.

As in Theorem~\ref{nomeasure}, it suffices to show that $\mathcal{V}$ is indestructible by forcing with any $\mathbb{M}_\kappa$. So, let $G$ be generic for $\mathbb{M}_\kappa$ over $V$. We will work in $V[G]$. Let $\dot{S} \in \dot{\mathcal{S}}$
(here we are referring to the interpretations of these names in the generic extension, but retaining the checks and dots so as to avoid confusion). Let $\dot{f} \in \dot{\mathcal{X}}$ be such that $\dot{S}(n) \subseteq 2^n\setminus \{\dot{f}(n)\}$
for every $n$. Then, corresponding to $\dot{f}$ there is an infinite and co-infinite subset $\dot{X}$  of $\omega$ such that for every $n$, $\dot{X}\cap [2^n, 2^{n+1}) = 1$. By the indestructibility of $\check{\mathcal{D}}$, it follows that there
	is an infinite $\check{Y}\subseteq \dot{X}$ such that $\check{Y} \in \check{\mathcal{D}}$. Corresponding to $\check{Y}$, there is some infinite $\check{T} \in \check{\mathcal{V}}$ such that $\check{T}(n)\subseteq \{\dot{f}(n)\}$ for every $n$. Then
	$\check{T}\nsubseteq^* \dot{S}$ and so it follows that $\dot{S}$ cannot localise $\check{\mathcal{V}}$. 

Altogether, $\mathcal{V}$ is indestructible by forcing with any $\mathbb{M}_\kappa$, and so $\mathfrak{T}_\mathcal{V}/Fin$ does not support a measure.

\end{proof}

\begin{thm} The Boolean algebra $\mathfrak{T}_\mathcal{V}/Fin$ has property (*).
\end{thm}

\begin{proof} As before, it suffices to find a countable covering $\mathcal{V} = \bigcup_n \mathcal{C}_n$ such that for every $n$ and every infinite $\mathcal{C} \subseteq \mathcal{C}_n$, there is an $S \in \mathcal{S}$ such that for infinitely many $T \in \mathcal{C}$, $T \subseteq S$.

So, first we get the sets $\mathcal{C}_n$. For each $A \in \mathcal{V}$ we fix a pair $(k_A, U_A)$ such that $k_A\in \omega\setminus\{0\}$ is such that for every $m \geq k_A$ we have $\frac{1}{2^m}|A(m)| < \frac{1}{9}$, and such that $U_A = A \cap
([0, k_A)\times \omega)$. As there are only countably many such pairs which are admissible, and each element is associated with some such pair, we get a countable partitioning $\mathcal{V} = \bigcup_n \mathcal{C}_n$ such that any two elements of $\mathcal{V}$ in the same piece of the partition agree about the pair. We claim that this countable partitioning witnesses property (*).

	Fix some $n'$, and the pair $(k,U)$ associated with $\mathcal{C}_{n'}$. Let $\mathcal{C} \subseteq \mathcal{C}_{n'}$ be a countably infinite subset, and we assume that we have fixed some enumeration of its elements, $\mathcal{C}=\{S_0, S_1, \ldots\}$. 

We shall need the following observation.

\begin{claim}
	Let $Q\subseteq \omega$ be infinite, $I\subseteq \omega$ finite such that $k \leq I$, and $n\in Q$. Then there is $T\in \mathcal{V}$,
	$T\subseteq (I \times \omega)$ and an infinite $Q'\subseteq Q$ including $n$ such that 
	\[ \forall m\in Q' \ \forall j\in I \ S_m(j)\subseteq T(j), \]
	and for each $j \in I$, $\frac{1}{2^j}|T(j)| < \frac{1}{3}$.
\end{claim}

\begin{why}
	Note that the set of $T\in \mathcal{V}$, such that 
\begin{enumerate}
    \item $T\subseteq I\times \omega$,
	\item $T(j)\subseteq 2^j$ for each $j\in I$,
    \item for every $j \in I$, $\frac{1}{2^j}|T(j)| < \frac{1}{3}$,
	\item $S_n \cap (I\times \omega) \subseteq T$
\end{enumerate}
is finite. But for every $m \in Q$, there is some such $T$ so that $S_m\cap (I\times \omega) \subseteq T$ (because $\frac{1}{2^j}|S_m(j)|<1/9$ for each $j\geq k$). So one of these $T$ must work for infinitely many $m \in Q$, including $n$. 
\end{why}

Now, we would like to find an infinite $N \subseteq \omega$ such that $\bigcup \{S_n \colon n \in N\} \in \mathcal{S}$. We shall build $N$ inductively, and we shall need to keep track of some extra information during this inductive construction. We will construct $(n_i,k_i,T_i,Q_i)_{i\in\omega}$ such that for every $i$
\begin{enumerate}
	\item $k_i\in\omega$ and $k_{i+1}>k_i$,
	\item $T_i \subseteq [k_i,k_{i+1})\times \omega$, $\frac{1}{2^j} |T_i(j)| < \frac{1}{3}$ for every $j\in [k_i,k_{i+1})$,\label{T_n}
	\item $Q_i\in [\omega]^\omega$, $Q_{i+1}\subseteq Q_i$,
	\item $n_{i}\in Q_{i+1}$, $n_{i+1}>n_i$,\label{i_n}
	\item $S_l \cap ([k_i,k_{i+1})\times \omega) \subseteq T_i$ for each $l\in Q_{i+1}$,\label{m>n}
		\item $\frac{1}{2^j}|S_{n_i}(j)| < \frac{1}{3^{i+1}}$ for each $j>k_{i+1}$.\label{k_n}
\end{enumerate}

Let $n_0 = 0$, $k_0=k$ and $Q_0 = \omega$ and suppose that we have constructed $n_i$, $k_i$ and $Q_i$. Then use the fact that $S_{n_{i}}\in \mathcal{V}$ to pick an appropriate $k_{i+1}>k_i$ to satisfy condition (\ref{k_n}). Then apply the claim to $Q=Q_i$, $I = [k_i,
	k_{i+1})$ and $n=n_i$ to obtain $T_i = T$ and $Q_{i+1} = Q'$. Now, take $n_{i+1}\in Q_{i+1}$ such that $n_{i+1}>n_i$ and proceed.

	Now, let $N = \{n_i\colon i\in\omega\}$ and
	\[V = \bigcup \{S_{n} \colon n\in N\}. \]

	We are done once we show that $V\in \mathcal{S}$. For this, we only need to check that for every $j \in \omega$, $|V(j)| < 2^j$. This is clear if $j \in [0,k)$, and otherwise there is some $i$ such that $j \in [k_i, k_{i+1})$. Of course,
		\[ V(j) = \bigcup_{m < i} S_{n_m}(j) \cup \bigcup_{m\geq i} S_{n_m}(j). \]
		By (\ref{m>n}) and (\ref{i_n}) $S_{n_m}(j) \subseteq T_i(j)$ for every $m\geq i$. So, by (\ref{T_n}), $|\bigcup_{m\geq i} S_{n_m}(j)|<2^j \cdot \frac{1}{3}$. 
		By (\ref{k_n}) 
		\[ |\bigcup_{m<i} S_{n_m}(j)| < 2^j \cdot \sum_{m=0}^{i-1} \frac{1}{3^{m+1}} <  2^j \cdot \frac{2}{3}, \]
		and we are done.
\end{proof}

The reason for our interest in property (*) is that by \cite[Lemma 2D]{Fremlin-Maharam}, every \emph{Maharam algebra} (i.e., complete Boolean algebra with a strictly positive exhaustive submeasure) satisfies this chain condition. Maharam's Problem, asking if there is a Maharam
algebra which does not support a measure, was a longstanding open problem (see \cite{Maharam}) in measure theory. In \cite{Talagrand} Talagrand gave an example of such an algebra. Given the complexity of Talagrand's algebra, it is natural to search for simpler examples, for example one not appealing to the existence of a non-principal ultrafilter on $\omega$.

Of course the algebra $\mathfrak{T}_\mathcal{V}/Fin$ is not
complete but instead of $\mathfrak{T}_\mathcal{V}/Fin$ we can consider its completion $\mathfrak{A}$. The algebra $\mathfrak{A}$ does not support a measure (otherwise $\mathfrak{T}_\mathcal{V}/Fin$ would support a measure). Moreover, since $\mathfrak{T}_\mathcal{V}/Fin$ is a $\pi$-base of
$\mathfrak{A}$, the algebra $\mathfrak{A}$ is $\sigma$-$n$-linked for every $n$ and satisfies property (*). By \cite[Theorem 1]{Todorcevic04} every complete Boolean algebra which is $\sigma$-$n$-linked and weakly distributive is a Maharam algebra. In \cite{Dobrinen04}
Dobrinen proved that several complete $\sigma$-linked algebras which do not support a measure are not weakly distributive by showing that they add a Cohen real. Unfortunately, the algebra $\mathfrak{A}$ also adds a Cohen real, so it cannot be weakly
distributive

\begin{thm} Forcing with $\mathfrak{T}_\mathcal{V}/Fin$ adds a Cohen real.\label{Cohen}
\end{thm}

\begin{proof} Clearly, if we can prove that forcing with the dense subposet formed by a $\pi$-base of $\mathfrak{T}_\mathcal{V}/Fin$ adds a Cohen real, then the result follows. The $\pi$-base we choose is $\{T_A \cap T_{(T,n)}\colon A \in \mathcal{V}, (T,n)
\in \Omega)\}/Fin\setminus\{[\emptyset]\}$. We consider this set with the inclusion relation. 

It shall, however, be convenient to consider some canonical representatives of $[T_A\cap T_{(S,n)}]$, where $T_A\cap T_{(S,n)}$ is infinite. Given such $T_A\cap T_{(S,n)}$, let $m \geq n$ be the least natural number such that $|A(m)|<2^m-1$. Notice that $[T_A\cap T_{(S,n)}] =[T_A\cap T_{(T,m)}]$ where $T= A \cap (m\times 2^m)$. Furthermore, it is easy to verify that if $B \in \mathcal{V}$ and $(U,l) \in \Omega$ are such that $|B(l)|<2^l-1$ and $[T_A \cap T_{(T,m)}] = [T_B \cap T_{(U,l)}]$, then $A=B$, $m=l$ and $T=U$. It follows that the poset 
\[ \mathbb{Q} = \{T_A \cap T_{(T,n)}\colon T_A \cap T_{(T,n)}\not \in \mathrm{Fin}, A\in \mathcal{V}, (T,n)\in \Omega, |A(n)| < 2^n-1\},\] 
with the order given by $\subseteq^*$ is isomorphic to our $\pi$-base. The dense subset of $\mathbb{Q}$ we shall consider is the following:
\[\poset = \{T_A \cap T_{(T,n)}\in \mathbb{Q}\colon n >1\}.\]

For technical reasons, the particular incarnation, $\coh$, of Cohen forcing that we will work with is the poset which consists of finite partial functions from $\omega\setminus 2$ to $2$, with the order being reverse inclusion.

Now, in order to prove that forcing with $\poset$ adds a Cohen real, it suffices (see \cite{Abraham-Aronszajn}) to define a projection from $\poset$ to $\coh$, that is, a function $ \Phi :\poset \to \coh$ such that 
\begin{enumerate}
    \item $\Phi[\poset]$ is a dense subset of $\coh$,
    \item $\Phi$ is order-preserving,
    \item If $p \in \poset$ and $\tau \in \coh$ is such that $\tau\leq \Phi(p)$, then there is a $q \leq p$ in $\poset$ such that $\Phi(q) = \tau$.
\end{enumerate}

Let $G$ be the generic filter of $\cB$. Note that we can define from $G$ an element of $\mathcal{S}$ by noting that in $V[G]$ the set \[\bigcup\{A(n) \colon A \in \mathcal{V}, \exists (S,m) \in \Omega, m > n, T_A\cap T_{(S,m)} \in G\}\] has size less than $2^n$ for every $n \in \omega$. Let $\dot{H}$ be a name for this `generic slalom'. 

We shall now define for each natural number $n>1$ a function $d_n\colon \mathcal{P}(2^n) \to \{0,1\}$. For $n\in \omega \setminus 2$ let 
$$
d_n(F)=
\begin{cases}
	1 \mbox{  if } 2^{n-1} \subseteq F\\
	0 \mbox{  otherwise.}
\end{cases}
$$
Note that if $A \subseteq 2^n$ has size less than $2^{n-1}$, then there are $B, C \in [2^n]^{2^{n}-1}$ such that $A \subseteq B \cap C$ and $d_n(B) = 1- d_n(C)$. 

The crucial property of the functions $\{d_n\}_{n >1}$ is the following: given any $T_A\cap T_{(S,n)}\in \poset$, since $A \in \mathcal{V}$ the set 
\[F = \{m >1\colon\exists i \ T_A\cap T_{(S,n)}\Vdash``d_m(\dot{H}(m)) = i"\}\]
is finite. To see this notice that if $M \geq n$ is such that for every $m\geq M$, $\frac{1}{2^m}|A(m)| < \frac{1}{2}$, then $F \subseteq M$. Indeed, for any $m \geq M$ in $\omega\setminus 2$ and $i \in 2$, we can find $B_i\subseteq 2^m$ containing $A(m)$ such that $d_m(B_i) = i$. For $i\in 2$ consider the slalom $A_i \in
\mathcal{V}$ such that $A_i(m') = A(m')$ for every $m' \neq m$, and $A_i(m) = B$. Then $p_i = T_{A_i} \cap T_{(S,n)} \subseteq^* T_A \cap T_{(S,n)}$ and $p_i \Vdash ``d_m(\dot{G}(m))=i"$ for each $i\in 2$ (here we have not shown that the $p_i$ are in $\poset$, so what we mean is that any extension of the $p_i$ in $\poset$ forces the respective statements). 

To finish, consider the function $\Phi\colon \poset \to \coh$ given by 
\[\Phi(T_A \cap T_{(S,n)}) = \{(m,i)\colon T_A \cap T_{(S,n)}\Vdash ``d_m(\dot{H}(m)) = i"\}.\]

We will check that $\Phi$ satisfies the desired properties. First, let $(S,2)\in \Omega$ and $A \in \mathcal{V}$ be such that $T_A \cap T_{(S,n)}$ is infinite, $|A(m)|\frac{1}{2^m} < \frac{1}{2}$ for every $m \in \omega$. 
Then $T_A\cap T_{(S,2)}\in \poset$ and $\Phi(T_A\cap T_{(S,2)}) = \emptyset$. 

Second, let $T_A \cap T_{(S,n)} \in \poset$ be such that $\Phi(p)= \sigma$, and let $\tau \leq \sigma$. Notice that if $k \in [2, n)$, then $k \in \mathrm{dom}(\sigma)$, so if $k \in \mathrm{dom}(\tau)\setminus \mathrm{dom}(\sigma)$, then $k \geq n$.
	Now, let $B\in \mathcal{V}$ be given by $B(k)=A(k)$ for every $k \not \in \mathrm{dom}(\tau)\setminus\mathrm{dom}(\sigma)$, and otherwise $B(k)$ is such that $d_k(B(k)) = \tau(k)$. Let $m$ be the least such that $m \geq n$ and $m \not \in
	\mathrm{dom}(\tau)$. Note that this implies that $m \not \in \mathrm{dom}(\sigma)$, $B(m)= A(m)$, and also that $|A(m)| < 2^m-1$. Let $T$ be arbitrary such that $T \cap (n \times 2^n) = S$, $(T,m)\in \Omega$, and $T_B \cap T_{(T,m)}$ is infinite.
	Then $T_B \cap T_{(T,m)} \in \poset$, is below $T_A \cap T_{(S,n)}$, and $\Phi(T_B \cap T_{(T,m)})= \tau$. Note that, as there is $q\in \mathbb{Q}$ such that $\Phi(q)=\emptyset(=1_\mathbb{C}$), this also gives us that $\Phi$ is a surjection. 

That $\Phi$ is order-preserving is clear.
\end{proof}

\begin{remark}
	In \cite[Theorem 6A]{Fremlin-Maharam} it is shown (and attributed there to Todor\v{c}evi\'{c}) that the \emph{Gaifman algebra} of \cite{Gaifman} is an example of a $\sigma$-linked Boolean algebra with property (*) which does not support a
	measure (and Dobrinen has also shown in \cite{Dobrinen04} that this algebra adds a Cohen real). It contains an uncountable independent family, similarly to
	$\mathfrak{T}_\mathcal{V}/Fin$. Using techniques from Section \ref{destruct} we can show that consistently we can produce such an example without a big independent family. 
	
	Recall that 
	\[ \cof^*(\mathcal{I}) = \min \{|\mathcal{A}|\colon \mathcal{A}\subseteq \mathcal{I}, \ \forall I\in \mathcal{I} \ \exists A\in \mathcal{A} \ I\subseteq^* A\}. \]
	Notice that assuming $\mathrm{add}(\mathcal{N})=\mathrm{cof}^*(\mathcal{D})$ and using the technique of the proof of Theorem \ref{kunen-fremlin2} we can find an $\subseteq^*$-chain $\mathcal{A}\subseteq \mathcal{V}$ such that each element of $\mathcal{V}$
	is almost contained in an element of $\mathcal{A}$. Then we can argue as in Proposition \ref{nomeasure2} to show that the algebra $\mathfrak{T}_\mathcal{A}/Fin$ does not support a measure.

	Finally, in this way we can construct a Boolean algebra which does not contain an uncountable independent family, which does not support a measure, but which is $\sigma$-$n$-linked for each $n$ and has property (*). 
\end{remark}

Recall that for an ideal $\mathcal{I}$, we denote its dual filter by $\mathcal{I}^*$. That is, $\mathcal{I}^* = \{I^c\colon I\in \mathcal{I}\}$. 

\begin{remark}
	If we are interested rather in the completion of $\mathfrak{T}_\mathcal{V}/Fin$ than in $\mathfrak{T}_\mathcal{V}/Fin$ itself, then we could present it in a slightly different language. Recall that the Mathias forcing, $\mathbb{M}(\mathcal{F})$, for a filter $\mathcal{F}$ on $\omega$ consists of all pairs $(s,F)$ such that $s\subseteq \omega$ is a finite set, $F\subseteq (\max s, \infty)$ and $F\in \mathcal{F}$. The ordering is given by $(s',F') \leq (s,
	F)$ if 
	\begin{enumerate}
		\item $s\subseteq s'$,
		\item $F'\subseteq F$,
		\item $s'\setminus s \subseteq F$.
	\end{enumerate}
	The forcing $\mathbb{M}(\mathcal{F})$ is $\sigma$-centered and it adds generically a pseudointersection of $\mathcal{F}$. 

We will consider a sub-poset $\mathbb{P}(\mathcal{F})$ of $\mathbb{M}(\mathcal{F})$ by imposing one more restriction on the conditions:
\begin{enumerate}
	\item[(4)] for every $n \in \omega$, $(s\cup F) \cap [2^n,2^{n+1}) \ne \emptyset$.
\end{enumerate}


We will show that 
$\mathrm{RO}(\mathbb{P}(\mathcal{D}^*))$ is isomorphic to the completion of $\mathfrak{T}_\mathcal{V}/Fin$. The difference is simply one of viewpoint: whereas a generic for $\mathfrak{T}_\mathcal{V}/Fin$ defines a member of $\mathcal{S}$, a generic for $\mathbb{P}(\mathcal{D}^*)$ will define a subset of $\mathcal{X}$ whose complement is a member of $\mathcal{S}$.



It is enough to find a poset isomorphic to a $\pi$-base of $\mathfrak{T}_\mathcal{V}/Fin$ which is order-isomorphic to a dense subset of $\poset(\mathcal{D}^*)$.

By the same argument as in Theorem~\ref{Cohen}, the set 
\[ \mathbb{Q} = \{T_A \cap T_{(T,n)}\colon T_A \cap T_{(T,n)}\not \in \mathrm{Fin}, A\in \mathcal{V}, (T,n)\in \Omega, |A(n)| < 2^n-1\},\]
under the order given by the reverse almost-inclusion is isomorphic to a $\pi$-base of $\mathfrak{T}_\mathcal{V}/Fin$. The advantage again being that if $A\in \mathcal{V}$ and $(S,n)\in \Omega$ witness that $T_A \cap T_{(S,n)} \in \mathbb{Q}$, then they are uniquely determined by $[T_A \cap T_{(S,n)}]$.

The order-isomorphism $\varphi\colon \mathbb{Q} \to \mathbb{P}(\mathcal{D}^*)$ is induced by the natural enumeration $f$ of $\{(n,i)\colon i< 2^n\}$ sending $\{n\}\times 2^n$ to $[2^n,2^{n+1})$ for each $n$ in the following way: for $A\in \mathcal{V}$ and $(S,n)\in \Omega$ witnessing that $T_A \cap T_{(S,n)} \in \mathbb{Q}$, let
	\[ \varphi(T_A\cap T_{(S,n)}) = (2^n \setminus f[S],f[A]^c \setminus 2^n). \] 
	
To see that $\varphi$ is an order embedding, that the range of the embedding is 
\[\{(s,F) \in \mathbb{P}(\mathcal{D}^*) \colon \exists n \in \omega, s\subseteq 2^n, F \subseteq [2^n, \infty), |F \cap [2^n, 2^{(n+1)})| >1\},\]
and that this range is actually a dense subset of $\mathbb{P}(\mathcal{D}^*)$, we use that if $T_A \cap T_{(S,n)} \subseteq^* T_B \cap T_{(T,m)}$ are distinct elements of $\mathbb{Q}$, then 
\begin{enumerate}
    \item $n \geq m$,
    \item $A \supseteq B$,
    \item $S \cap (m \times 2^m) = T$,
    \item $S \cap ([m,n)\times 2^n) \supseteq A \cap ([m,n)\times 2^n) \supseteq B \cap ([m,n)\times 2^n)$.
\end{enumerate}

	So,  we have that $\mathbb{P}(\mathcal{D}^*)$ is an example of a complete Boolean algebra adding a Cohen real which is not $\sigma$-centered and which does not support a measure, but is
$\sigma$-$n$-linked for each $n$ and satisfies property (*).  Analogously,
$\mathbb{P}(\mathcal{I}_{1/n}^*)$ is a complete Boolean algebra supporting a measure which is not $\sigma$-centered and which adds a Cohen real (the last statement can be proved in the same way as Theorem \ref{Cohen}). 

Perhaps using other ideals one can construct in this way other interesting Boolean algebras. 
\end{remark}

\section{Acknowledgements} 

This research was done whilst the authors were visiting fellows at the Isaac Newton Institute for Mathematical Sciences, Cambridge, in the
programme `Mathematical, Foundational and Computational Aspects of the Higher Infinite' (HIF). We would like to thank Tomek Bartoszy\'{n}ski, Andreas Blass, Mirna D{\v z}amonja and Osvaldo Guzman for their helpful remarks concerning the subject of this paper.

\bibliographystyle{alpha}
\bibliography{smallb}

\end{document}

%% file: command.tex
\newcommand{\con}{\mathfrak c}

\newcommand{\om}{\omega}

\newcommand{\coh}{\mathbb C}
\newcommand{\poset}{\mathbb P}

\newcommand{\cB}{\mathcal B}

\newcommand{\cof}{{\rm cof}}

\DeclareMathOperator{\fin}{Fin}

\newcommand{\Pw}[1]{\mathcal{P}(#1)}
\newcommand{\pom}{\Pw{\om}}
\newcommand{\pofin}{\pom / \fin}

\newtheorem{thm}{Theorem}[section]
\newtheorem{defi}[thm]{Definition}
\newtheorem{lem}[thm]{Lemma}
\newtheorem{cor}[thm]{Corollary}
\newtheorem{prop}[thm]{Proposition}
\newtheorem{fact}[thm]{Fact}
\newtheorem{prob}[thm]{Problem}

\newtheorem*{claim}{Claim}

\theoremstyle{remark}
\newtheorem{remark}[thm]{Remark}

\newenvironment{why}[1][Proof]{\proof[#1]\mbox{}}{\endproof} 
\newcommand{\azero}{\aleph_0}